\documentclass[a4paper]{article}%for notes: {scrartcl}
\usepackage{amsthm,amsmath,amssymb}
\usepackage{enumerate,enumitem}
\usepackage{graphicx}
\usepackage{xfrac}
\usepackage{ulem}
\usepackage{mathtools}
\usepackage{ wasysym }%Bowtie%gemini
\usepackage{tabularx,booktabs,multirow}
\usepackage{hyperref}
\newcolumntype{C}{>{\centering\arraybackslash}X}
\newcolumntype{D}{>{\centering\arraybackslash}X}

\DeclareGraphicsExtensions{.pdf,.png,.eps}
\graphicspath{{Figures/}}

\newtheorem{theorem}{Theorem}
\newtheorem{lemma}[theorem]{Lemma}

\newtheorem{proposition}[theorem]{Proposition}

\newtheorem{question}[theorem]{Question}

\newtheorem*{claim*}{Claim}

\theoremstyle{definition}

\newtheorem{construction}[theorem]{Construction}

\newcommand{\cE}{\ensuremath{\mathcal{E}}}
\newcommand{\cG}{\ensuremath{\mathcal{G}}}
\newcommand{\cH}{\ensuremath{\mathcal{H}}}
\newcommand{\cV}{\ensuremath{\mathcal{V}}}

\newcommand{\N}{\ensuremath{\mathbb{N}}}

\usepackage[usenames,dvipsnames]{xcolor}

\newcommand{\footremember}[2]{%
  \footnote{#2}
  \newcounter{#1}
  \setcounter{#1}{\value{footnote}}%
}
\newcommand{\footrecall}[1]{%
  \footnotemark[\value{#1}]%
}

\begin{document}

\title{A note on asymmetric hypergraphs}

\author{Dominik Bohnert \footremember{grant}{Karlsruhe Institute of Technology, Karlsruhe, Germany.}
\and Christian Winter \footrecall{grant}\ \footnote{E-mail: \textit{christian.winter@kit.edu}} }

\maketitle

\begin{abstract}
A $k$-graph $\cG$ is asymmetric if there does not exist an automorphism on $\cG$ other than the identity, and $\cG$ is called minimal asymmetric if it is asymmetric but every non-trivial induced sub-hypergraph of $\cG$ is non-asymmetric.
Extending a result of Jiang and Ne\v set\v ril~\cite{JN}, we show that for every $k$-graph, $k\ge3$, there exist infinitely many minimal asymmetric $k$-graphs which have maximum degree $2$ and are linear. Further, we show that there are infinitely many $2$-regular asymmetric $k$-graphs for $k\ge3$.
\end{abstract}

%%%%%%%%%%%%%%%%%%%%%%%%
\section{Introduction}

For $k\ge 2$, a \textit{$k$-uniform hypergraph}, or \textit{$k$-graph}, is a pair $\cG=(\cV(\cG),\cE(\cG))$ such that the edge set $\cE(\cG)$ consists of $k$-element subsets of the vertex set $\cV(\cG)$. Note that $2$-graphs are commonly known as graphs.
An \textit{automorphism} on a $k$-graph $\cG=(\cV,\cE)$ is a bijection $\phi\colon \cV\to \cV$ such that for every $E\in\cE$, $\{\phi(v):~ v\in E\}\in\cE$.
An automorphism which is not the identity is called \textit{non-trivial}.
We say that a $k$-graph $\cG$ is \textit{symmetric} if there exists a non-trivial automorphism on $\cG$ and \textit{asymmetric} otherwise.
$\cG$ is \textit{minimal asymmetric} if it is asymmetric and every induced sub-hypergraph $\cH$ of $\cG$ with $2\le |\cV(\cH)|<|\cV(G)|$ is symmetric. 
\\

Asymmetry of graphs was first considered by Frucht \cite{Frucht49} in 1949. It was famously observed by Erd\H{o}s and R\'enyi \cite{ER} that almost all graphs are asymmetric.
In 1988, Ne\v set\v ril conjectured that the number of minimal asymmetric graphs is finite, see \cite{BRSSTW}. 
After several partial results \cite{Nesetril,  Sabidussi, NS}, this conjecture was recently confirmed by Schweitzer and Schweitzer \cite{SS} who showed that there are exactly $18$ minimal asymmetric graphs.
%For further results on asymmetry of graphs, see ....
In the hypergraph setting, Ellingham and Schroeder \cite{ES} studied a related variant of asymmetry. The notion of asymmetry used in this paper was introduced by Jiang and Ne\v set\v ril \cite{JN}, who showed that the natural generalization of Ne\v set\v ril's conjecture to $k$-graphs does not hold.

\begin{theorem}[Jiang, Ne\v set\v ril \cite{JN}]\label{thm:JN}
Let $k\ge 3$ be a positive integer. Then there exist infinitely many minimal asymmetric $k$-graphs.
\end{theorem}

They provide an explicit construction in their proof where each $k$-graph has maximum degree $3$.
In this paper, we extend their result by requiring even stricter structural properties.
In a $k$-graph $\cG=(\cV,\cE)$ for any two distinct vertices $u,v\in\cV$ the \textit{codegree} of $u$ and $v$ is the number of edges in $\cE$ which contain both $u$ and $v$. The \textit{maximum codegree} of $\cG$ is the maximum over the codegrees of all vertex pairs $u,v\in\cV$, $u\neq v$.
A $k$-graph $\cG$ with maximum codegree $1$ is called \textit{linear}. Equivalently, $\cG$ is linear if any two edges $E_1,E_2\in\cE$ intersect in at most one vertex. Here we prove the following strengthening of Theorem \ref{thm:JN}.
%A $k$-graph $\cG=(\cV,\cE)$ is said to be \textit{linear} if any two edges $E_1,E_2\in\cE$ intersect in at most one vertex, i.e.\ $|E_1\cap E_2|\le 1$.
%Equivalently, linear $k$-graphs can be characterized as those $k$-graphs where every vertex pair have co-degree at most $1$.

\begin{theorem}\label{thm:main}
Let $k\ge 3$ be a positive integer. There exist infinitely many minimal asymmetric $k$-graphs which have maximum degree $2$ and maximum codegree $1$.
\end{theorem}

Note that every $k$-graph with maximum degree $1$ or maximum codegree $0$ is symmetric, so our result is best possible in this sense.
In our construction for Theorem \ref{thm:main}, most vertices have degree $2$, but crucially some vertices have degree $1$.
This raises the question whether there exist (minimal) asymmetric $k$-graphs where every vertex has the same degree.
We say that a $k$-graph $\cG$ is \textit{$r$-regular} if every vertex has degree $r$.

\begin{theorem}\label{thm:regular}
There are infinitely many $2$-regular, asymmetric $k$-graphs for every $k\ge3$.
\end{theorem}

It remains open if this result extends to minimal asymmetric $k$-graphs. We raise the following unanswered question.

\begin{question}
For $k\ge 3$ and $r\ge2$, is there an $r$-regular, minimal asymmetric $k$-graph?
\end{question}

Note that this question can be answered negatively for $k=2$ and arbitrary $r$: None of the $18$ minimal asymmetric graphs characterized by Schweitzer and Schweitzer \cite{SS} is regular.
\\

In this paper we use standard graph theoretic notions, for formal definitions we refer the reader to Diestel \cite{Diestel}.
We denote by $[n]$ the set of the first $n$ integers $\{1,\dots,n\}$. For consistency, let $[0]=\varnothing$.
Given a function $\phi:\cV\to\cV$ and a subset $W\subseteq \cV$, we denote image of $W$ by $\phi(W):=\{\phi(v):~ v\in W\}$.

The organization of this paper is as follows. In Section \ref{sec:JN} we present the constructions needed for the proof of Theorem \ref{thm:main}, in Sections \ref{sec:properties} and \ref{sec:connectivity} we show some properties of these constructions.
Subsequently, in Section \ref{sec:proof} we prove Theorem \ref{thm:main} and in Section \ref{sec:regular} we give a proof of Theorem \ref{thm:regular}.

\section{Sparse minimal asymmetric $k$-graphs}\label{sec:sparse}

\subsection{Constructions}\label{sec:JN}

The following construction was given by Jiang and Ne\v set\v ril \cite{JN}. Throughout the section, all indexes in $[tk]$ are considered modulo $tk$.

\begin{construction}[Jiang, Ne\v set\v ril \cite{JN}]\label{constr:JN}
Let $k\ge 3$ and $t\ge 2$. Let $\cG_{k,t}$ be the $k$-graph with vertices
$$\cV(\cG_{k,t})=\big\{u_i : ~ i\in[tk]\big\} \cup \big\{v_i : ~ i\in[tk]\big\} \cup \big\{w_{i,j} : ~ i\in[tk], j\in[k-3]\big\}$$
and edges $\cE(\cG_{k,t})=\cE_L\cup\cE_{cyc}$. 
Here $\cE_L=\big\{ E_i : ~ i\in[tk] \big\}$ is the set of \textit{L-edges}
$$E_i=\{v_i,u_i,v_{i+1},w_{i,1},\dots, w_{i,k-3}\}.$$
Furthermore, the set of \textit{cyclic edges} is $\cE_{cyc}=\big\{E_{i,j}: ~ j\in[k-3], i=j+sk, s\in[t] \big\} $ where
$$E_{i,j}=\{w_{i,j},\dots,w_{i+k-1,j}\}.$$
An illustration of this construction is given in Figure \ref{fig:JNgraph}.
\end{construction}

\begin{figure}[h]
\centering
\includegraphics[scale=0.50]{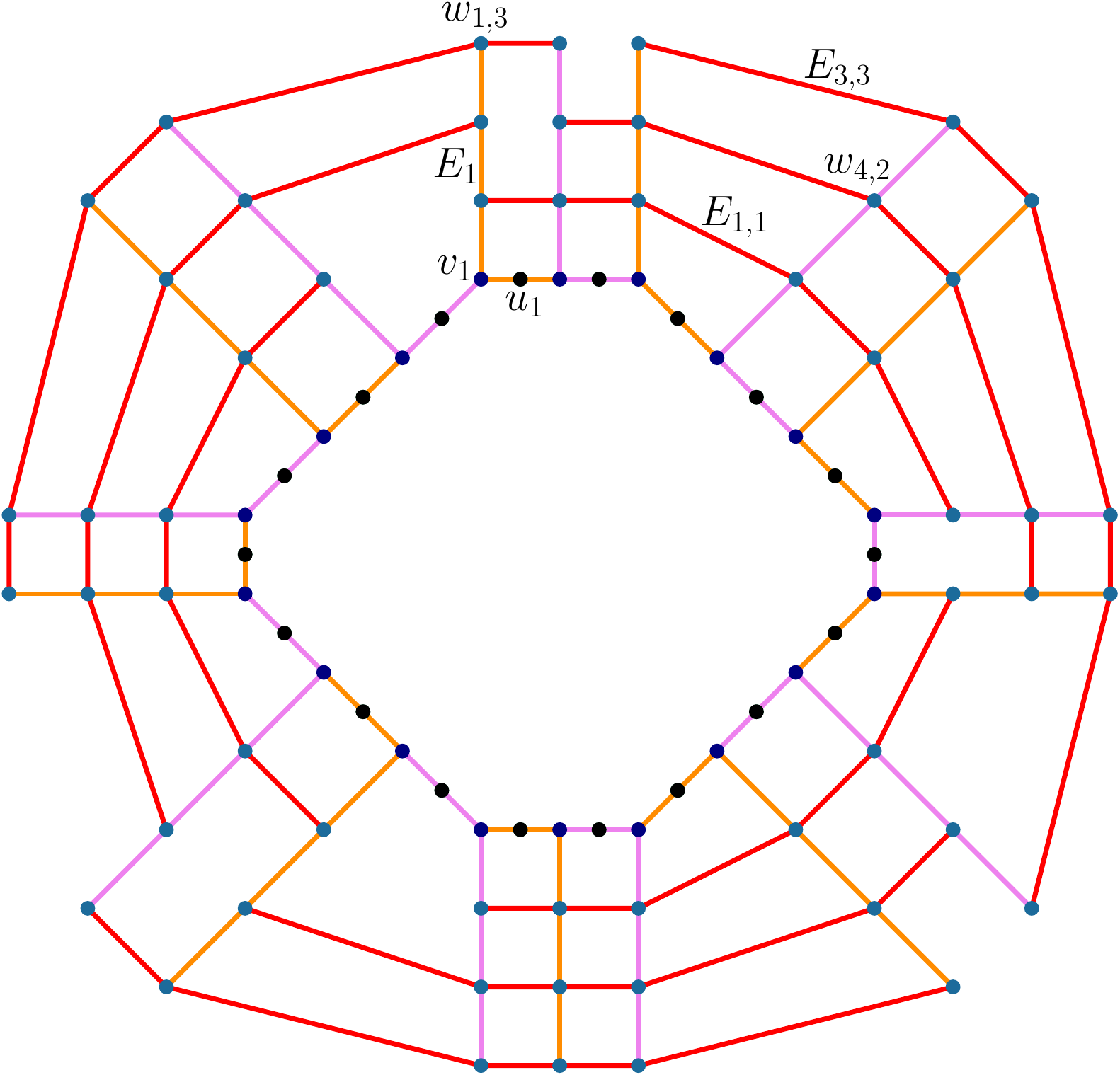}
\caption{The $3$-graph $\cG_{6,3}$}
\label{fig:JNgraph}
\end{figure}

Jiang and Ne\v set\v ril \cite{JN} proved Theorem \ref{thm:JN} by adding a single edge to $\cG_{k,t}$.
In this paper, we extend Construction \ref{constr:JN} as follows.

\begin{construction}\label{constr:H}
Let $k\ge 3$ and let $t_i\in \N$ for $i\in[k-1]$ such that $2\le t_1<t_2<\dots<t_{k-1}$.
We denote by $\cG^\ell=(\cV^\ell,\cE^\ell)$ a copy of $\cG_{k,t_\ell}$ as introduced in Construction  \ref{constr:JN}, such that the vertex sets $\cV^\ell$ are pairwise disjoint. For every $\ell\in[k-1]$, we write $u^\ell_i$ when referring to the vertex of $\cG^\ell$ corresponding to $u_i$ in $\cG_{k,t_\ell}$ and similarly for $v^\ell_i, w^\ell_{i,j}, E^\ell_i$ and $E^\ell_{i,j}$. 

Now let $x_0$ be an additional vertex which is not contained in any $\cV^\ell$, $\ell\in[k-1]$.
We define $\cH(t_1,\dots,t_{k-1})=(\cV,\cE)$ such that
$$\cV=\cV^1\cup\dots\cup\cV^{k-1}\cup\{x_0\}\quad \text{ and }\quad \cE=\cE^1\cup\dots\cup \cE^{k-1}\cup\{E_0\},$$
where $E_0=\{x_0, u^1_1,u^2_1,\dots,u^{k-1}_1\}$. See Figure \ref{fig:Hgraph} for an illustration of $\cH(t_1,\dots,t_{k-1})$.
\\

\begin{figure}[h]
\centering
\includegraphics[scale=0.65]{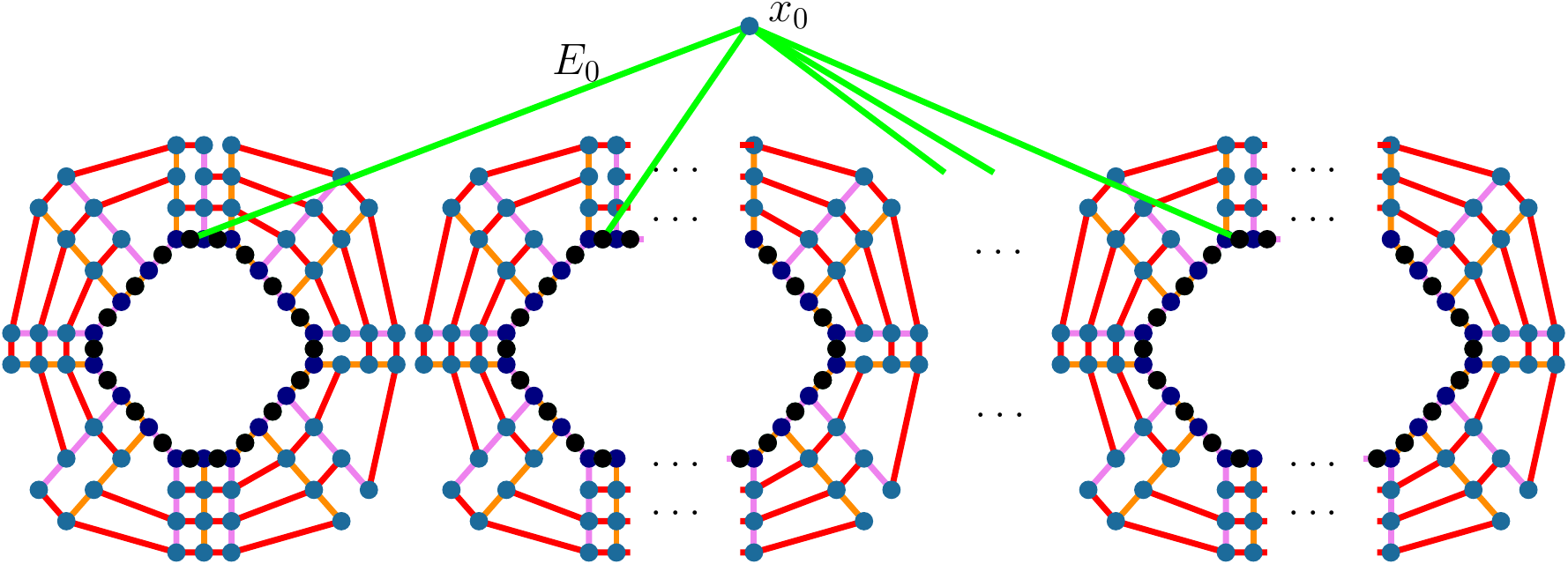}
\caption{The $k$-graph $\cH(t_1,\dots,t_{k-1})$}
\label{fig:Hgraph}
\end{figure}
\end{construction}

The $k$-graph $\cH(t_1,\dots,t_{k-1})$ is non-asymmetric if $k=3$ or $k=5$, because Lemma \ref{lem:invariant} does not hold for such $k$, see also Figure \ref{fig:counterex5}. Therefore, we provide two additional constructions covering those cases.

\begin{construction}\label{constr:H3}
Let $k\in\{3,5\}$ and $2\le t< t'$. Let $\cG$ and $\cG'$ be vertex-disjoint copies of $\cG_{k,t}$ and $\cG_{k,t'}$, respectively. 
We denote by $u'_i$ the vertex corresponding to $u_i$ in $\cG_{k,t'}$ and similarly for $v'_i, w'_{i,j}, E'_i$ and $E'_{i,j}$.
For the vertices in $\cG$ we use the same labels as defined for $\cG_{k,t}$, e.g.\ $u_i$ refers to the vertex in $\cG$ corresponding to $u_i$ in $\cG_{k,t}$.
Let $x_0,y$ and $y'$ be three distinct vertices, disjoint from $\cV(\cG)\cup\cV(\cG')$.
\\

\noindent For $k=3$, let $E_0=\{x_0,u_1,u'_1\}$, $E_y=\{y, u_2,u_3\}$ and $E'_y=\{y',u'_2,u'_3\}$.
We define the $3$-graph
$$\cH^3(t,t')=\big( \cV(\cG)\cup\cV(\cG')\cup\{x_0,y,y'\},\cE(\cG)\cup\cE(\cG')\cup\{E_0,E_y,E'_y\}\big).$$
For $k=5$, let $E_0=\{x_0,u_1,u_2,u'_1,u'_2\}$ and define the $5$-graph
$$\cH^5(t,t')=\big( \cV(\cG)\cup\cV(\cG')\cup\{x_0\},\cE(\cG)\cup\cE(\cG')\cup\{E_0\}\big).$$
Both constructions are illustrated in Figure \ref{fig:H5}.

\begin{figure}[h]
\centering
\includegraphics[scale=0.65]{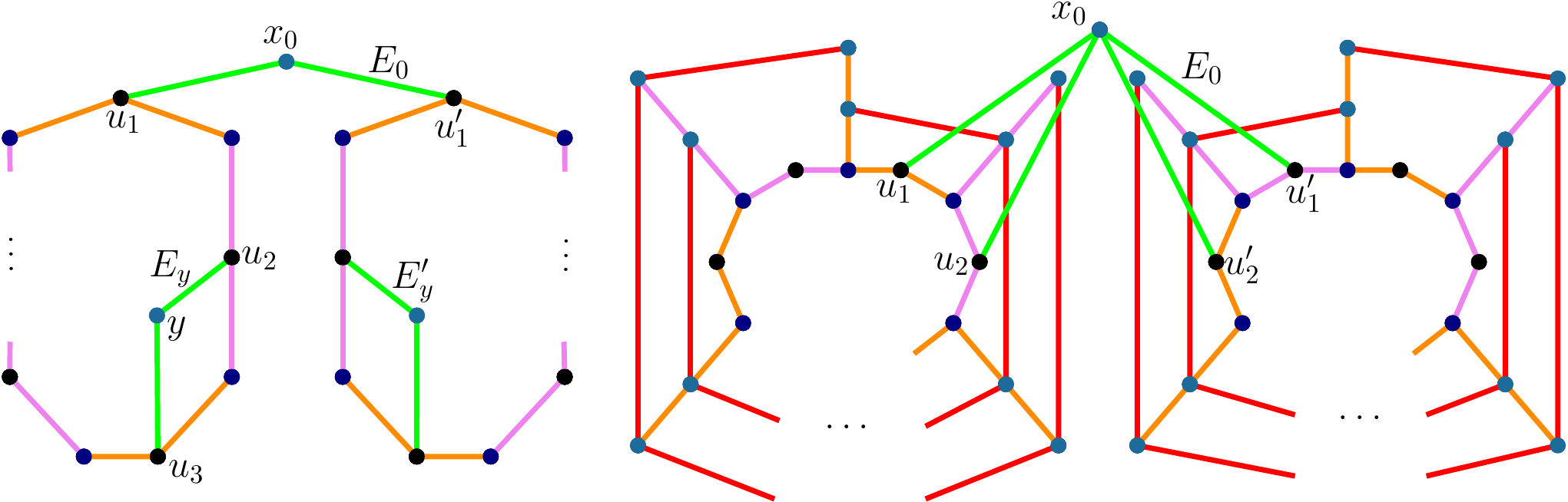}
\caption{The hypergraphs $\cH^3(t,t')$ and $\cH^5(t,t')$}
\label{fig:H5}
\end{figure}

\end{construction}

\subsection{Properties of $\cG_{k,t}$}\label{sec:properties}

First we state slight reformulations of two properties shown by Jiang and Ne\v set\v ril \cite{JN}.
%Jiang and Ne\v set\v ril \cite{JN} showed the following properties.

\begin{lemma}[Jiang, Ne\v set\v ril \cite{JN}]\label{lem:subgraph}
Let $k\ge 3$ and $t\ge 2$. Let $\cG'$ be an induced sub-hypergraph of $\cG_{k,t}$ on at least two vertices.
\begin{enumerate}
\item[(i)] There is a non-trivial automorphism on $\cG'$, i.e.\ $\cG'$ is symmetric.
\item[(ii)] If $E_1\in \cE(\cG')$ and $|\cV(\cG')|<|\cV(\cG_{k,t})|$, then there  is a non-trivial automorphism $\phi$ on $\cG'$ such that $\phi(E_1)=E_1$ and $\phi(u_1)=u_1$.
\end{enumerate}
\end{lemma}

A stronger version of Lemma \ref{lem:subgraph}(i) is given in Lemma 4(2) of \cite{JN}, where \textit{weak} sub-hypergraphs are considered. 
Lemma \ref{lem:subgraph}(ii) follows from the proof of Lemma 3(3) of \cite{JN}.
\\

\begin{lemma}\label{lem:basicproperties}%[Jiang, Ne\v set\v ril \cite{JN}]
Let $k\ge 3$ and $t\ge 2$. Let $\phi$ be an automorphism on $\cG_{k,t}$. 
\begin{enumerate}
\item[(i)] Then $\{\phi(u_i) : ~ i\in[tk] \}=\{u_i : ~ i\in[tk] \}$. Furthermore, $\phi(E)\in\cE_L$ for every $E\in\cE_L$ and $\{\phi(v_i) : ~ i\in[tk] \}=\{v_i : ~ i\in[tk] \}$.
\item[(ii)] There is a $j\in[tk]$ such that either $\phi(E_i)=E_{i+j-1}$ for every $i\in[tk]$ or $\phi(E_i)=E_{j-i+1}$ for every $i\in[tk]$, where the indexes are considered modulo $tk$.
\end{enumerate}
\end{lemma}

We remark that the statement of Lemma \ref{lem:basicproperties}(i) is given implicitly in~\cite{JN}.
Statements similar to Lemma \ref{lem:basicproperties}(ii) appear as Lemma 3(1) in \cite{JN} and as Lemma~9(1) in~\cite{JN2}. 
Unfortunately, both versions only cover one of the two characterizations of $\phi$ accurately.

\begin{proof}[Proof of Lemma \ref{lem:basicproperties}]
Note that the $u_i$'s are exactly the vertices of degree $1$ in $\cG_{k,t}$. Observe that, since $\phi$ is an automorphism, $v$ and $\phi(v)$ have the same degree for every vertex $v$, thus $\phi(\{u_i : ~ i\in[tk] \})=\{u_i : ~ i\in[tk] \}$. This implies (i).

A consequence of (i) is that $\phi(E_1)=E_{j}$ for some $j\in[tk]$, so $\{\phi(v_1),\phi(v_2)\}=\{v_{j},v_{j+1}\}$.
If $\phi(v_1)=v_{j}$ and $\phi(v_2)=v_{j+1}$, then $\phi(E_2)=E_{j+1}$ and thus $\phi(v_3)=v_{j+2}$.  
Iteratively, we find that $\phi(E_i)=E_{i+j-1}$ for every $i\in[tk]$. 
Now suppose that $\phi(v_1)=v_{j+1}$ and $\phi(v_2)=v_{j}$. Then $\phi(E_2)=E_{j-1}$ and iteratively $\phi(E_i)=E_{j-i+1}$ for every $i\in[tk]$.
This completes the proof of (ii).
\end{proof}

\begin{lemma}\label{lem:invariant}
Let $k=4$ or $k\ge 6$ and $t\ge 2$. Let $\phi$ be an automorphism on $\cG_{k,t}$.
If $\phi(E_1)=E_1$, then $\phi$ is the identity.
\end{lemma}

A result closely related to Lemma \ref{lem:invariant} is given in Lemma 9(2) of \cite{JN2} without a proof. 
Note that Lemma 3(2) of \cite{JN} almost corresponds to our Lemma \ref{lem:invariant}, but it does not hold for $k=3$ and $k=5$, see for example the non-trivial automorphism illustrated in Figure \ref{fig:counterex5}.

\begin{proof}[Proof of Lemma \ref{lem:invariant}]
Assume that $\phi$ is not the identity. Then Lemma \ref{lem:basicproperties} provides that $\phi(E_i)=E_{2-i}$ for every $i\in[tk]$. 

If $k=4$, then Lemma \ref{lem:basicproperties}(i) implies that $\{\phi(w_{i,1}): ~ i\in[tk] \}=\{w_{i,1} : ~ i\in[tk] \}$.
Since $\phi(E_1)=E_1$ and $\phi(E_2)=E_{tk}$, we find that $\phi(w_{1,1})=w_{1,1}$ and $\phi(w_{2,1})=w_{tk,1}$, respectively.
Observe that the edge $E_{1,1}$ contains vertices $w_{1,1}$ and $w_{2,1}$. 
However, there is no edge in $\cG_{k,t}$ containing $\phi(w_{1,1})=w_{1,1}$ and $\phi(w_{2,1})=w_{tk,1}$, a contradiction.

Now suppose that $k\ge 6$. Then consider the edge $E_{3,3}=\{w_{3,3},\dots,w_{k+2,3}\}$. It intersects each of the edges $E_3,\dots,E_{k+2}$.
Because $\phi$ is an automorphism, $\phi(E_{3,3})$ has a non-empty intersection with each of the edges $\{\phi(E_3),\dots,\phi(E_{2+k})\}=\{E_{(t-1)k},\dots,E_{tk-1}\}$.
However, it is easy to see from our construction that in $\cG_{k,t}$ such an edge does not exist.
\end{proof}
%\begin{figure}[h]
%\centering
%%\includegraphics[scale=0.6]{VLaN.pdf}
%\caption{Non-trivial automorphism $\phi$ on $\cG_{3,2}$}
%\label{fig:counterex3}
%\end{figure}
\begin{figure}[h]
\centering
\includegraphics[scale=0.53]{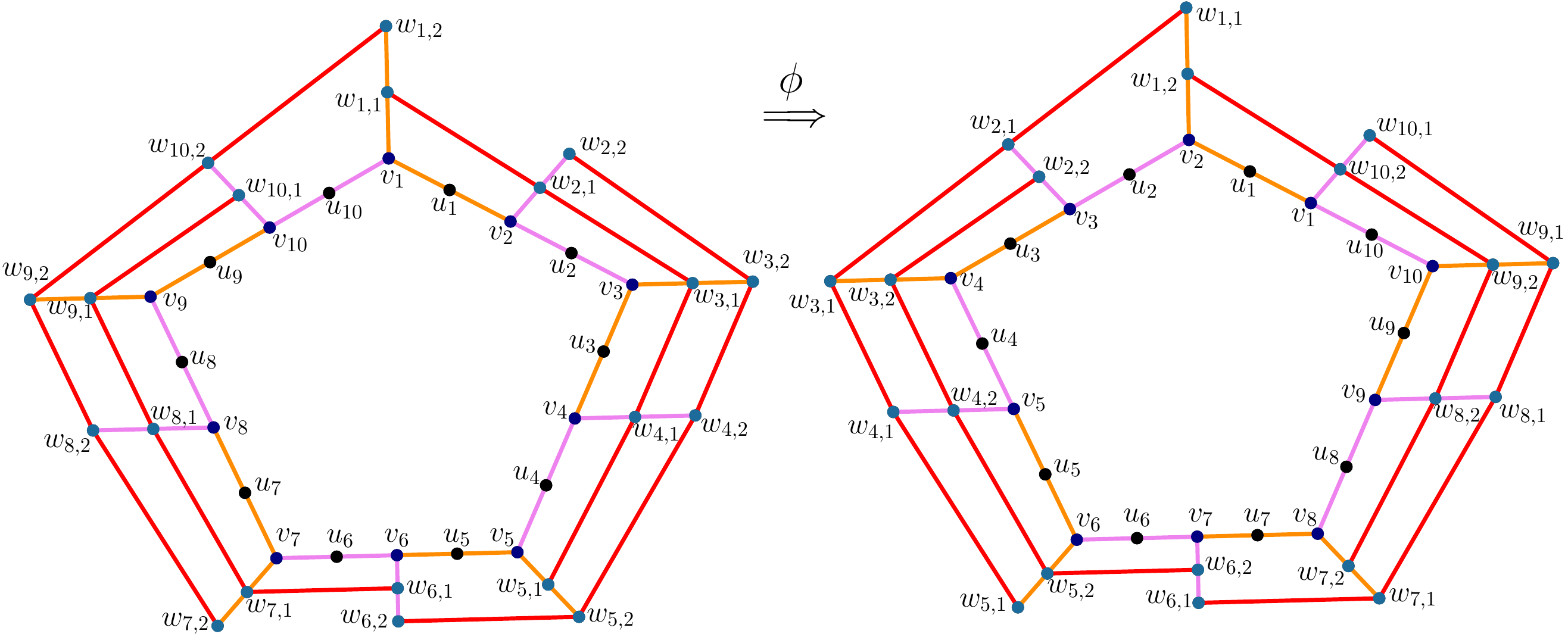}
\caption{Non-trivial automorphism $\phi$ on $\cG_{5,2}$ with $\phi(E_1)=E_1$}
\label{fig:counterex5}
\end{figure}

\subsection{Connectivity}\label{sec:connectivity}

Next we introduce a notion of \textit{connectivity} between two vertices of a $k$-graph.
In a $k$-graph $\cG$, a \textit{$v_1-v_{r+1}$-path} is an alternating sequence $(v_1,E_1,v_2,E_2,\dots,E_r,v_{r+1})$ of $r+1$ distinct vertices $v_i\in\cV(\cG)$ and $r$ distinct edges $E_i\in\cE(\cG)$ such that both $v_i$ and $v_{i+1}$ are contained in $E_i$ for any $i\in[r]$. 
Such paths are commonly known as \textit{Berge paths}.
Two $u-v$-paths are \textit{edge-disjoint} if the underlying edge sets of the paths are disjoint.
We say that $u$ and $v$ are \textit{$t$-connected} if there are $t$ pairwise edge-disjoint $u-v$-paths.
It is a simple observation that an automorphism leaves the connectivity invariant:
\begin{proposition}\label{prop:connected}
Let $\cG$ be a $k$-graph and $u,v\in\cV(\cG)$. Let $\phi$ be an automorphism on $\cG$. Then %the connectivity of $u$ and $v$ is invariant under $\phi$, i.e.\ 
$u$ and $v$ are $t$-connected if and only if $\phi(u)$ and $\phi(v)$ are $t$-connected.
\end{proposition}

\begin{lemma}\label{lem:connected}
Let $k\ge 3$ and $t\ge 2$ and consider $\cG:=\cG_{k,t}$.
Let $E\in\cE(\cG)$ and let $u,v\in E$ be distinct vertices of degree $2$, i.e.\ $u,v\notin\{u_1,\dots,u_{tk}\}$.
Then $u$ and $v$ are $2$-connected.
\end{lemma}

\begin{proof}
Consider $\cG'$, the $k$-graph obtained from $\cG$ by deleting the edge $E$. We shall show that $u$ and $v$ are $1$-connected in $\cG'$.
Recall that the edges set $\cE(\cG)=\cE_L\cup \cE_{cyc}$ consists of L-edges $E_i$ and cyclic edges $E_{i,j}$.

If $E\in\cE_{cyc}$, then $u$ and $v$ are contained in distinct L-edges of $\cG$, say without loss of generality $u\in E_1$ and $v\in E_j$. 
Then $(u,E_1,v_2,E_2,v_3,\dots,E_j,v)$ is a $u-v$-path in $\cG'$.

If $E\in\cE_L$, let $i$ such that $E=E_i$. Note that the L-edges of $\cG'$ form a $v_{i+1}-v_{i}$-path containing all $v_j$, $j\in[tk]$.
Therefore, if there is a $u-v_j$-path and a $v-v_{j'}$-path in $\cG'$ for any $j,j'\in[tk]$, then we also find a $u-v$-path in $\cG'$.
If $u\in\{v_1,v_2\}$, there is a trivial $u-v_i$-path. Otherwise, $u\in E'$ for some cyclic edge $E'\in\cE_{cyc}$. 
Let $w$ be an arbitrary vertex in $E'\setminus\{u\}$. Then $w$ is also contained in some L-edge $E_{j}\in \cE_L$ where $j\in[tk]$.
Then $(u,E',w,E_{j},v_{j})$ is a $u-v_{j}$-path. Similarly, we find a $v-v_{j'}$-path, which completes the proof.
\end{proof}

%\section{Proof of Theorem \ref{thm:main}}
\subsection{Proof of the main result}\label{sec:proof}

\begin{proof}[Proof of Theorem \ref{thm:main}]
For the first part of the proof, let $k=4$ or $k\ge 6$. Let $t_i\in \N$ for $i\in[k-1]$ such that $2\le t_1<t_2<\dots<t_{k-1}$.
We shall show that $\cH:=\cH(t_1,\dots,t_{k-1})$ is minimal asymmetric. 
In order to verify that $\cH$ is asymmetric, let $\phi$ be an arbitrary automorphism on $\cH$. 
Recall that $E_0$ is an edge of $\cH$ which connects otherwise disjoint copies of $\cG_{k,t_i}$, $i\in[k-1]$.
\\

First, we show that $\phi(E_0)=E_0$.
We know that $\phi(E_0)\in\cE(\cH)$, so assume that $\phi(E_0)=E$ for some $E\neq E_0$. 
Then $E\in\cE(\cG^\ell)$ for some $\ell\in[k-1]$.
Consider two distinct vertices $u,v\in E_0\setminus\{x_0\}$. Both vertices have degree $2$ in $\cH$.
Since $\phi$ is an automorphism, $\phi(u)$ and $\phi(v)$ are distinct vertices in $E$ with degree $2$.
By Lemma \ref{lem:connected}, $\phi(u)$ and $\phi(v)$ are $2$-connected in $\cG^\ell$, so in particular $2$-connected in $\cH$.
However, in our construction the vertices $u$ and $v$ are not $2$-connected, because the only path between them is $(u,E_0,v)$.
This contradicts Proposition \ref{prop:connected}. Therefore, we conclude $\phi(E_0)=E_0$, so in particular $\phi(x_0)=x_0$.
\\

Since $\phi(E_0)=E_0$, $\phi$ is also an automorphism on the $k$-graph $\cH-E_0$, which is the disjoint union of an isolated vertex $x_0$ and $k$-graphs $\cG^\ell$. 
Recall that the $t_\ell$'s are pairwise disjoint, therefore the $\cG^\ell$'s are pairwise non-isomorphic.
This implies that $\{\phi(E): ~ E\in \cE(\cG^\ell)\} = \cE(\cG^\ell)$ for every $\ell\in[k-1]$, i.e.\ $\phi$ maps each $\cG^\ell$ to itself.

Now we show that $\phi$ is the identity, thus $\cH$ is asymmetric.
Fix an arbitrary $\ell\in[k-1]$.
Note that $\phi(u^\ell_1)\in \phi(E_0)\cap \phi(E^\ell_1)$, thus $\phi(u^\ell_1)\in E_0\cap \cV(\cG^\ell)$.
This implies that $\phi(u^\ell_1)=u^\ell_1$ and therefore $\phi(E^\ell_1)=E^\ell_1$.
Now Lemma \ref{lem:invariant} provides that $\phi$ restricted to $\cG^\ell$ is the identity. 
Since $\ell$ was chosen arbitrarily, the entire automorphism $\phi\colon \cH\to\cH$ is the identity.
We conclude that $\cH$ is asymmetric.
\\

Now let $\cH'$ be an arbitrary induced sub-hypergraph of $\cH$ with $2\le|\cV(\cH')|<|\cV(\cH)|$.
We shall show that that $\cH'$ is symmetric. We can assume that $|\cE(\cH')|\ge2$, otherwise $\cH'$ is trivally symmetric.
\\

\textbf{Case 1:} $E_0\notin \cE(\cH')$.\smallskip\\
Let $E\in\cE(\cH')$ with $E\neq E_0$, i.e.\  $E\in\cE(\cG^\ell)$ for some fixed $\ell\in[k-1]$.
Let $\cH''$ be the sub-hypergraph of $\cH'$ induced by the vertex set $\cV(\cG^\ell)$.
Because $E\in\cE(\cH'')$, $\cH''$ has at least two vertices.
Now Lemma \ref{lem:subgraph}(i) provides a non-trivial automorphism $\psi$ on $\cH''$.
We extend this automorphism to $\cH'$ as follows. Let $\phi:\cV(\cH')\to\cV(\cH')$ with $\phi(w)=\psi(w)$ for every $w\in \cV(\cH'')$ and $\phi(w)=w$ for every $w\notin \cV(\cH'')$.
%$$\phi(w)=\begin{cases}\psi(w), &\text{ if } w\in \cV(\cH''),\\ w, &\text{ otherwise}.\end{cases}$$
Then $\phi$ is a non-trivial automorphism on $\cH'$, so $\cH'$ is symmetric.
\\

\textbf{Case 2:} $E_0\in \cE(\cH')$.\smallskip\\
If there is some $\ell\in[k-1]$ such that $E^\ell_1\notin\cE(\cH')$, then consider the function $\phi:\cV(\cH')\to\cV(\cH')$ with
$\phi(u^\ell_1)=x_0$, $\phi(x_0)=u^\ell_1$ and $\phi(w)=w$ for every $w\in \cV(\cH')\setminus\{u^\ell_1,x_0\}$. 
Observe that this is a non-trivial automorphism on $\cH'$.

If $E^{\ell'}_1\in\cE(\cH')$ for every $\ell'\in[k-1]$, fix $\ell$ such that there is a vertex $v\in\cV(\cG^\ell)\setminus\cV(\cH')$ and let $\cH''$ be the sub-hypergraph of $\cH'$ induced by $\cV(\cG^\ell)$. 
Then $2<|\cV(\cH'')|<|\cV(\cG^\ell)|$, thus Lemma \ref{lem:subgraph}(ii) yields a non-trivial automorphism $\psi$ on $\cH''$ with $\psi(u^\ell_1)=u^\ell_1$.
Similarly to Case~1, we extend $\psi$ to a non-trivial automorphism on $\cH'$. 
\\

This completes the proof for $k=4$ and $k\ge 6$. For $k=3$ and $k=5$, let $t$ and $t'$ be arbitrary integers with $2\le t< t'$.
We show that $\cH^3:=\cH^3(t,t')$ and $\cH^5:=\cH^5(t,t')$ are minimal asymmetric. The proof is similar to the argumentation presented above, so we only provide a sketch. A detailed proof is given in the first author's thesis \cite{Bohnert}.
\\

If $k=3$, let $\phi$ be an arbitrary automorphism on $\cH^3$. Then $\phi(E_0)=E_0$, and thus
$\{\phi(E): ~ E\in \cE(\cG)\} = \cE(\cG)$ as well as $\{\phi(E): ~ E\in \cE(\cG')\} = \cE(\cG')$.
Therefore, $\phi(u_1)=u_1$ and $\phi(u'_1)=u'_1$. There are six edges in which every vertex has degree $2$, namely $E_i$ and $E'_i$ for $i\in[3]$. 
It is easy to see that each of them is invariant under $\phi$, which then implies that $\phi$ is the identity.
Thus, $\cH^3$ is asymmetric. 

Now consider an induced sub-hypergraph $\cH'$ of $\cH^3$ with $2\le|\cV(\cH')|<|\cV(\cH^3)|$.
We can suppose that there is no edge which contains two vertices of degree $1$, otherwise $\cH'$ is clearly symmetric.
If $E_0\notin \cE(\cH')$, then there is a non-trivial automorphism $\phi$ on $\cH'$ with $\phi(u_2)=u_3$ and $\phi(u_3)=u_2$.
If $E_0\in \cE(\cH')$, then $E_1,E'_1\in\cE(\cH')$. Let $E\in\cE(\cH^3)\setminus \cE(\cH')$ and say that $E\in \cE(\cG)$. 
Now if $E_{y}\notin \cE(\cH')$, we can apply Lemma \ref{lem:subgraph}(ii) as in Case 2, so suppose that $E_y\in\cE(\cH')$.
Since no edge contains two vertices of degree $1$, we find that $\cE(\cH')\cap\cE(\cG)=\{E_1,E_2,E_3,E_y\}$.
Then there is an automorphism $\phi$ on $\cG$ with $\phi(E_y)=E_3$ and $\phi(E_3)=E_y$.
\\

If $k=5$, given an automorphism $\phi$ on $\cH^5$, we see that $\phi(E_0)=E_0$, thus $\{\phi(u_1),\phi(u_2)\}=\{u_1,u_2\}$.
If $\phi(u_1)=u_2$ and $\phi(u_2)=u_1$, then $\phi(E_1)=E_{2}$, $\phi(E_2)=E_1$ and $\phi(E_3)=E_{tk}$. Note that the edge $E_{1,1}$ intersects each of $E_1,E_2$ and $E_3$,
but there does not exist an edge in $\cE(\cH^5)$ which intersects $\phi(E_1)=E_{2}, \phi(E_2)=E_1, \phi(E_3)=E_{tk}$, a contradiction.
Thus $\phi(u_1)=u_1$ and $\phi(u_2)=u_2$, and similarly, $\phi(u'_1)=u'_1$ and $\phi(u'_2)=u'_2$.
By Lemma \ref{lem:basicproperties}(ii), $\phi$ is the identity, so $\cH^5$ is asymmetric. 

Let $\cH'$ be an induced sub-hypergraph of $\cH^5$ such that $2\le|\cV(\cH')|<|\cV(\cH^5)|$.
If $E_0\notin \cE(\cH')$, we proceed as in Case 1.
Otherwise, we can suppose that $E_1,E_2,E'_1,E'_2\in\cE(\cH')$ by a similar argument as in Case 2. 
Then a variant of Lemma \ref{lem:subgraph}(ii), see Lemma 2.9(i) of \cite{Bohnert}, provides a non-trivial automorphism on $\cH'$.
\end{proof}

\section{Regular asymmetric $k$-graphs}\label{sec:regular}

In order to prove Theorem \ref{thm:regular}, we need a result by Izbicki \cite{Izbicki}.

\begin{theorem}[Izbicki \cite{Izbicki}]\label{thm:izbicki}
For every $k\ge 3$, there exist infinitely many $k$-regular asymmetric $2$-graphs.
\end{theorem}

Given a $k$-regular $2$-graph $G=(\cV,\cE)$, the \textit{(hypergraph) dual} of $G$ is the $k$-graph $\cH=(\cE,\{A(v): ~ v\in \cV\})$
where $A(v)=\{E\in \cE: v\in E\}$ is the \textit{adjacency set} of $v$.
Note that $|A(v)|=k$ for every $v$, thus $\cH$ is a well-defined $k$-graph. An example for a hypergraph dual is provided in Figure \ref{fig:dual}.
Observe that adjacency sets are unique in regular $2$-graphs:

\begin{proposition}\label{prop:adjacency}
Let $G$ be an $r$-regular $2$-graph, $r\ge2$. Let $u,v\in\cV(G)$ be two distinct vertices of $G$. Then $A(u)\neq A(v)$.
\end{proposition}

\begin{figure}[h]
\centering
\includegraphics[scale=0.33]{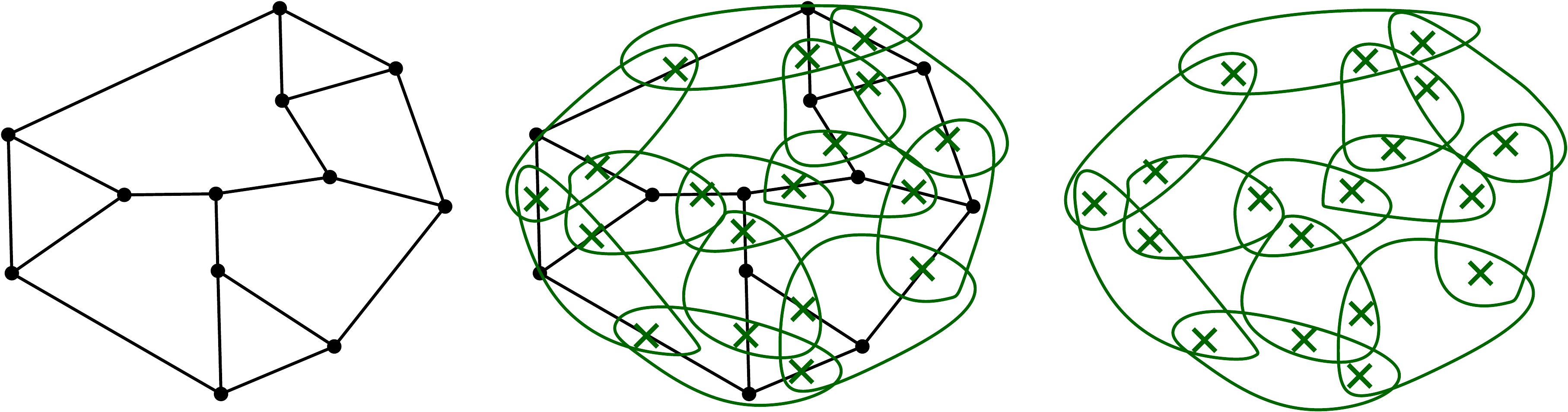}
\caption{The $3$-regular Frucht graph (left) and its hypergraph dual (right).}
\label{fig:dual}
\end{figure}

\begin{lemma}\label{lem:edgeasym}
Let $G$ be an $r$-regular asymmetric graph for some $r\ge3$. Then the hypergraph dual of $G$ is also asymmetric.
\end{lemma}

\begin{proof}Let $G=(\cV,\cE)$. 
Let $\phi_\cH$ be an arbitrary automorphism on the dual $\cH$ of $G$. %We call this function an \textit{edge-automorphism} of $G$. 
By the definition of a dual we know that $\phi_\cH\colon \cE\to\cE$ is a bijection such that for any $v\in\cV$, 
the edges in its adjacency set $A(v)$ are mapped to $\big\{\phi_\cH(E): ~ E\in A(v)\big\}=A(w_v)$ for some vertex $w_v\in\cV$.
By Proposition \ref{prop:adjacency}, $w_v$ is uniquely determined. 
\\

We define the function $\phi_G\colon \cV\to\cV$, $\phi_G(v)=w_v$. Observe that $\phi_G$ is a bijection. Now we show that $\phi_G$ is an automorphism on $G$.
Consider an arbitrary edge $E_{uv}=\{u,v\}\in\cE$.
Then $E_{uv}\in A(v)$, thus $\phi_\cH( E_{uv})\in \big\{\phi_\cH(E): ~ E\in A(v)\big\}=A(w_v)=A(\phi_G(v))$.
Similarly, we obtain $\phi_\cH( E_{uv})\in A(\phi_G(u))$.
Therefore, 
\begin{equation}
\phi_\cH( E_{uv})=\{\phi_G(u),\phi_G(v)\}.
\end{equation}
This implies that $\{\phi_G(u),\phi_G(v)\}\in\cE$, so $\phi_G$ is an automorphism on $G$. 
Since $G$ is asymmetric, $\phi_G$ is the identity. 
Then for every edge $E_{uv}=\{u,v\}\in\cE$, (1) implies that $\phi_\cH( E_{uv})=\{u,v\}=E_{uv}$, i.e.\ $\phi_\cH$ is the identity. 
Consequently, $\cH$ is asymmetric. 
\end{proof}

\begin{proof}[Proof of Theorem \ref{thm:regular}]
Let $k\ge 3$. Let $G_1,G_2,\dots$ be an infinite family of pairwise distinct $k$-regular, asymmetric graphs as provided by Theorem \ref{thm:izbicki}.
Let $\cH_i$ be the hypergraph dual of $G_i$, $i\in\N$. Observe that the $\cH_i$'s are pairwise distinct $2$-regular $k$-graphs. 
Lemma \ref{lem:edgeasym} provides that the $\cH_i$'s are asymmetric.%, thus we obtain infinitely many $2$-regular, asymmetric $k$-graphs.
\end{proof}

\textbf{Acknowledgements:} The research of the second author was partially supported by DFG grant FKZ AX 93/2-1. The authors thank Maria Axenovich and Torsten Ueckerdt for helpful discussions as well as Felix Clemen for comments on the manuscript.

\end{document}